\documentclass[11pt,reqno]{amsart}
\usepackage{amssymb,array,mathdots,setspace,tikz}
\usepackage[
	colorlinks=true,
	citecolor=blue,
	linkcolor=blue,
	urlcolor=blue]{hyperref}
\usepackage{setspace}
\usepackage[margin=1.2in]{geometry}
\usetikzlibrary{matrix}

\newcommand{\CC}{\mybb{C}}
\newcommand{\e}{\widetilde{e}}
\newcommand{\f}{\widetilde{f}}
\newcommand{\floor}[1]{\left\lfloor { #1 } \right\rfloor}
\newcommand{\g}{\mathfrak{g}}
\newcommand{\hM}{\widehat{\mathcal{M}}}
\newcommand{\MM}{\mathcal{M}}
\newcommand{\mybb}[1]{\mathbf{#1}}
\newcommand{\one}{\boldsymbol{1}}

\newcommand{\sage}{\textsc{SageMath}}
\newcommand{\wt}{{\rm wt}}
\newcommand{\ZZ}{\mybb{Z}}

\definecolor{darkred}{rgb}{0.7,0,0}
\newcommand{\defn}[1]{{\color{darkred}\emph{#1}}}

\theoremstyle{plain}
\newtheorem{thm}{Theorem}[section]
\newtheorem{lemma}[thm]{Lemma}

\newtheorem{prob}[thm]{Problem}
\newtheorem{prop}[thm]{Proposition}

\theoremstyle{definition}
\newtheorem{dfn}[thm]{Definition}
\newtheorem{ex}[thm]{Example}
\newtheorem{remark}[thm]{Remark}
\numberwithin{equation}{section}
\numberwithin{figure}{section}
\numberwithin{table}{section}

\newenvironment{acknowledgements}
{\bigskip\noindent\footnotesize\textbf{Acknowldegements.} }
{\medskip}

\newcommand{\GYW}[1]{
 \foreach \x [count=\s from 0] in {#1}{
   \foreach \y [count=\t from 0] in \x {
     \node[font=\scriptsize] at (-\t,\s) {$\y$};
     \draw (-\t+.5,\s+.5) to (-\t-.5,\s+.5);
     \draw (-\t+.5,\s-.5) to (-\t-.5,\s-.5);
     \draw (-\t-.5,\s-.5) to (-\t-.5,\s+.5);
   }
 \draw[-,thick] (.5,\s+1) to (.5,-.5) to (-\t-1,-.5);
 }
}

\newcommand{\wall}[1]{
        \begin{tikzpicture}[baseline=20,scale=.45]
                \GYW{#1}
        \end{tikzpicture}
}

\usepackage{listings}
\lstdefinelanguage{Sage}[]{Python}
{morekeywords={False,True},sensitive=true}
\begin{document}
\title{The weight function for monomial crystals of affine type}
\author{Luke James}
\author{Ben Salisbury}
\address{Department of Mathematics, Central Michigan University, Mt. Pleasant, MI 48859}
\email[Luke James]{james4la@cmich.edu}
\email[Ben Salisbury]{salis1bt@cmich.edu}
\urladdr[Ben Salisbury]{http://people.cst.cmich.edu/salis1bt}
\thanks{B.S.\ was partially supported by CMU Early Career grant \#C62847 and Simons Foundation grant \#429950.}
\subjclass[2010]{}

\maketitle

\begin{abstract}
In this work, an expression for the affine weight is calculated for Nakajima monomial crystals in affine types $A_n^{(1)}$ and $B_3^{(1)}$.
\end{abstract}

\section{Introduction}

In 2003, both Nakajima \cite{Nak:03} and Kashiwara \cite{Kash:03} defined a crystal structure on a certain set of monomials $\MM$, that have since been referred to as Nakajima monomials.  Using this crystal structure, it was shown that irreducible highest weight crystals can be modeled using Nakajima monomials \cite{Kash:03}.  Specifically, the irreducible highest weight crystal $B(\lambda)$ is isomorphic to the connected component of the crystal graph of all Nakajima monomials containing a highest weight monomial of weight $\lambda$.  Later, Kang--Kim--Shin \cite{KKS:07} modified the crystal structure of Nakajima monomials from \cite{Kash:03} and dubbed the modification the set of modified Nakajima monomials, denoted $\hM$.  It was shown that the connected component of the crystal graph of all modified Nakajima monomials containing the element $\one \in \hM$ is isomorphic to the crystal $B(\infty)$.

Given a modified Nakajima monomial $M = \prod_{i\in I}\prod_{k\ge 0} Y_{i,k}^{y_{i,k}}\one$ (see Section \ref{subsec:monomials} for an explanation of the notation), the weight of $M$ is defined to be 
\[
\wt(M) = \sum_{i\in I} \biggl( \sum_{k\ge0} y_{i,k} \biggr) \Lambda_i,
\]
where $\Lambda_i$ is the $i$th fundamental weight of the underlying Kac--Moody algebra $\g$.  If $\g$ is of finite type, then this description of the weight map is complete.  However, if $\g$ is of affine type (for now, suppose $\g$ is not of type $A_{2n}^{(2)}$), then the affine weight lattice of $\g$ has the form
\[
P = \ZZ \Lambda_0 \oplus \ZZ\Lambda_1 \oplus \cdots \oplus \ZZ \Lambda_n \oplus \ZZ\delta,
\]
where $\delta$ is the minimal imaginary root in the root system of $\g$.  Therefore, the weight map on the monomial model for crystals of affine type should include a term involving $\delta$.  Indeed, there should be some $\ZZ$-valued function $D$ on monomials such that 
\[
\wt(M) =  \sum_{i\in I} \biggl( \sum_{k\ge0} y_{i,k} \biggr) \Lambda_i + D(M)\delta.
\]
The definition in \cite{KKS:07} of the weight function, however, implies that the crystals constructed there are $U_q'(\g) = U_q([\g,\g])$ crystals when $\g$ is affine type, since, in this case, the weight lattice is $P_{\mathrm{cl}} = P/\ZZ\delta$.

By the structure of the root system of an (untwisted) affine Kac--Moody algebra $\g$, the minimal imaginary root satisfies the identity $\delta = \theta + \alpha_0$, where $\theta$ is the unique highest root of the underlying finite type root system of $\g$.  Moreover, by the crystal axioms, $\wt(\f_iM) = \wt(M) - \alpha_i$ for all elements $M$ in the crystal and all $i\in I$.  Therefore, the coefficient of $\delta$ in $\wt(M)$ will decrease by $1$ for each instance of a $0$-arrow in the path from $\one$ to $M$ in the crystal graph, so the function $D$ is obtained by counting the number of $0$-arrows in a path from $\one$ to $M$.  This is the approach taken in the results of this work.

The main results of this work give a description of the function $D$ in types $A_n^{(1)}$ ($n\ge 1$) and $B_3^{(1)}$.  The outline of the paper is as follows.  In Section \ref{sec:crystals}, we give an overview of the theory of abstract crystals and introduce both the modified Nakajima monomial model and the model given by generalized Young walls.  (Generalized Young walls will be important for the proof of the main result in type $A_n^{(1)}$ for $n\ge 2$.)  In Section \ref{sec:typeA}, the main result in type $A_n^{(1)}$ is given, but it is split into two cases: $n=1$ and $n\ge 2$.  The main result for type $B_3^{(1)}$ is given in Section \ref{sec:typeB} and some discussion is given on generalizing the result to types with $n\ge 3$.  Finally, Section \ref{sec:hw} explains how the main results may be applied to irreducible highest weight crystals modeled by Nakajima monomials.

\section{Crystals}\label{sec:crystals}

\subsection{Generalities on affine crystals}

Let $I = \{0,1,\dots,n\}$ be an index set, and let $(C,\Pi,\Pi^\vee,P,P^\vee)$ be a Cartan datum of affine type; i.e.,
\begin{itemize}
\item $C = (C_{ij})_{i,j\in I}$ is a generalized Cartan matrix of affine type,
\item $\Pi = \{\alpha_i:i\in I\}$ is the set of simple roots,
\item $\Pi^\vee = \{h_i:i\in I\}$ is the set of simple coroots,
\item $P^\vee = \ZZ h_0 \oplus\cdots\oplus \ZZ h_n \oplus \ZZ d$ is the dual weight lattice,
\item $\mathfrak{h} = \CC\otimes_\ZZ P^\vee$ is the Cartan subalgebra,
\item and $P = \{ \lambda\in\mathfrak{h}^* : \lambda(P^\vee) \subset \ZZ \}$ is the weight lattice.
\end{itemize}
The simple roots and simple coroots are related via the Cartan matrix: $\alpha_j(h_i) = C_{ij}$.  The fundamental weights $\Lambda_i \in P$ are defined as $\Lambda_i(h_j) = \delta_{i,j}$ and $\Lambda_i(d) = 0$.  Define $P^+ = \{ \lambda \in P : \lambda(h_i) \ge 0 \text{ for all } i\in I\}$ to be the set of dominant integral weights.  Finally, the canonical pairing $\langle\ ,\ \rangle\colon P^\vee \times P \longrightarrow \ZZ$  is defined by $\langle h, \lambda \rangle = \lambda(h)$ for all $h\in P^\vee$ and $\lambda \in P$.

Let $\mathfrak{g}$ be the affine Kac-Moody algebra associated with this Cartan datum, and denote by $U_q(\mathfrak{g})$ the quantized universal enveloping algebra of $\mathfrak{g}$.  We will always assume that $\g$ is of untwisted affine type; that is, one of the types from \cite[Table Aff 1]{Kac:90}.
We denote the generators of $U_q(\mathfrak{g})$ by $e_i$, $f_i$ ($i\in I$), and $q^h$ ($h\in P^\vee)$.  The subalgebra of $U_q(\mathfrak{g})$ generated by $f_i$ ($i\in I$) will be denoted by $U_q^-(\mathfrak{g})$.  Also, let $U_q'(\g)$ be the subalgebra of $U_q(\g)$ generated by $e_i$, $f_i$, and $K_i^{\pm1}$ ($i\in I$), where $K_i = q^{s_ih_i}$ and $S = \mathrm{diag}(s_i \in \ZZ_{>0} : i\in I)$ is a diagonal matrix such that $SC$ is a symmetric matrix.  For more information, see, for example, \cite{HK:02}. 

The null root, or minimal imaginary root, of the root system of $\g$ is defined to be 
\[
\delta = d_0\alpha_0 + d_1\alpha_1 + \cdots + d_n\alpha_n,
\]
where the integers $\{d_i : i \in I \}$ are given in \cite[Table Aff 1]{Kac:90}.  Since we are only considering untwisted affine types, we have $d_0 = 1$.  Moreover, by \cite[Prop.~6.4]{Kac:90}, $\delta = \alpha_0 + \theta$, where $\theta$ is the highest root of the underlying finite-type Lie algebra $\g_0$ of $\g$.
Using the null root, the weight lattice $P$ may be expressed as 
\[
P = \ZZ \Lambda_0 \oplus \ZZ\Lambda_1 \oplus \cdots \oplus \ZZ \Lambda_n \oplus \ZZ\delta.
\]
In terms of the simple roots, we have
\[
\delta = 
\begin{cases}
\alpha_0 + \alpha_1 + \cdots + \alpha_n & \text{ if } \g = A_n^{(1)}, \\
\alpha_0 + \alpha_1 + 2\alpha_2 + \cdots + 2\alpha_n & \text{ if } \g = B_n^{(1)}.
\end{cases}
\]
See \cite{Kac:90} for the expression of $\delta$ in terms of the simple roots in other affine types.

\begin{dfn}
An \defn{abstract $U_q(\g)$-crystal} associated to the affine quantum group $U_q(\g)$ is a set $B$ together with the maps 
\[
\wt\colon B \longrightarrow P, \ \ \ \ 
\varepsilon_i,\varphi_i\colon B \longrightarrow \ZZ \sqcup \{-\infty\}, \ \ \ \ 
\e_i, \f_i \colon B \longrightarrow B \sqcup \{0\}, \ (i \in I)
\] 
such that for all $i \in I$ and $b \in B$,
\begin{enumerate}
\item $\varphi_i(b) = \varepsilon_i(b) + \langle h_i, \wt(b)\rangle$,
\item $\wt(\e_i b) = \wt(b) + \alpha_i$, $\wt (\f_i b) = \wt(b)- \alpha_i$,
\item $\varepsilon_i(\e_i b) = \varepsilon_i(b) - 1$, $\varphi(\e_i b) = \varphi_i(b) + 1$,
\item $\varepsilon_i(\f_i b) = \varepsilon_i(b) + 1$, $\varphi_i(\f_i b) = \varphi_i(b) -1$,
\item\label{item:graph} $\f_i b = b'$ if and only if $\e_i b' = b$ for $b,b' \in B$,
\item $\e_i b = \f_i b = 0$ if $ \varepsilon_i(b) = -\infty$.
\end{enumerate}
\end{dfn}

The operators $\e_i$ and $\f_i$ above are known as the \defn{Kashiwara operators}. Note that condition (\ref{item:graph}) implies $B$ is equipped with an an $I$-colored directed graph structure given by $b \overset{i}{\longrightarrow} b'$ if and only if $\f_i b=b'$. 
This graph is called the \defn{crystal graph} of $B$. 

\begin{dfn} 
A \defn{crystal morphism} $\psi\colon B_1 \longrightarrow B_2$ is a map $\psi\colon B_1\sqcup \{0\} \longrightarrow B_2 \sqcup \{0\}$ satisfying the following conditions:
\begin{enumerate}
\item $\psi(0) = 0$,
\item $\wt (\psi(b)) = \wt(b)$, $\varepsilon_i(\psi(b)) = \varepsilon_i(b)$, and $ \varphi_i(\psi(b)) = \varphi_i(b)$ for all $b \in B_1$ such that $\psi(b)\neq 0$,
\item $\f_i \psi(b) = \psi(b')$ if $b,b' \in B_1$ and $\f_i b = b'$.
\end{enumerate}
\end{dfn}

An \defn{isomorphism} of crystals is defined as a bijective morphism of crystals such that $\psi(\f_ib) = \f_i\psi(b)$ for all $b\in B_1$ and $i\in I$.  A morphism $\psi\colon B_1 \longrightarrow B_2$ is said to be an \defn{embedding} if $\psi$ induces an injective map $B_1 \sqcup \{0\} \longrightarrow B_2 \sqcup\{0\}$.  Moreover, an embedding $\psi$ is called \defn{full} if, for all $b\in B_1$ such that $\e_i\psi(b) \in B_2$, $\e_ib \in B_1$.

\begin{ex}
For each $\lambda \in P^+$, the crystal basis $B(\lambda)$, as defined in \cite{Kash:91}, of the irreducible highest weight $U_q(\g)$-module $V(\lambda)$ is an abstract $U_q(\g)$-crystal.  The unique element of $B(\lambda)$ of weight $\lambda$ is denoted by $u_\lambda$.
\end{ex}

\begin{ex}
The crystal basis $B(\infty)$, as defined in \cite{Kash:91}, of the negative half of the quantum group $U_q^-(\g)$ is an abstract $U_q(\g)$-crystal.  The unique element of $B(\infty)$ of weight $0$ is denoted by $u_\infty$.
\end{ex}

\begin{ex}
Let $\lambda \in P^+$.  Define $T_\lambda = \{t_\lambda\}$ to be the one-element abstract $U_q(\g)$-crystal whose operations, for all $i\in I$, are defined as
\[
\e_it_\lambda = \f_it_\lambda = 0, \ \ \ \ 
\varepsilon_i(t_\lambda) = \varphi_i(t_\lambda) = -\infty, \ \ \ \
\wt(t_\lambda) = \lambda.
\]
By \cite{Kash:93}, there is a full crystal embedding $B(\lambda) \lhook\joinrel\longrightarrow B(\infty) \otimes T_\lambda$, where $\otimes$ denotes the crystal tensor product defined in \cite{Kash:91}.  We will not require the general definition of the crystal tensor product, but rather focus on tensor products of the form $B\otimes T_\lambda$, where $B$ is some abstract crystal.  In this case, the crystal graphs $B\otimes T_\lambda$ and $B$ are naturally isomorphic as $I$-colored directed graphs using the map $b \otimes t_\lambda \mapsto b$, for all $b\in B$, but the weights of corresponding vertices differ by $\lambda$; that is, for $b\in B$, $\wt(b \otimes t_\lambda) = \wt(b) + \lambda$.
\end{ex}

\subsection{Modified Nakajima monomials}\label{subsec:monomials}

Let $Y_{i,k} \ (i \in I, k \in \ZZ)$ be formal commuting variables with an additional commuting variable $\one$. Define the \defn{modified Nakajima monomials} as the set $\hM$ of all monomials of the form
\[
M = \prod_{i \in I} \prod_{k \geq 0} Y_{i,k}^{y_{i,k}} \one,
\]
where $y_{i,k} \in \ZZ$ and $y_{i,k} = 0$ for all but finitely many $k$. For such an $M$, define
\begin{subequations}\label{eq:mon_cry_ops}
\begin{align} 
\wt(M) &= \sum_{i \in I}\biggl(\sum_{k \geq 0} y_{i,k}\biggr) \Lambda_i,\\
\varphi_i(M) &= \max \left\{ \sum_{j=0}^k y_{i,j} : k \in \ZZ_{\geq 0}\right\},\\
\varepsilon_i(M) &= \varphi(M) - \langle h_i, \wt(M) \rangle, \label{eq:oldep}\\
k_f = k_f(M) &= \min\left\{k \in \ZZ_{\geq 0} : \varphi_i(M) = \sum_{j=0}^k y_{i,j}\right\},\\
k_e = k_e(M) &= \max \left\{ k \in\ZZ_{\geq 0} : \varphi_i(M) = \sum_{j=0}^k y_{i,j}\right\}.
\end{align}
\end{subequations}
Next, choose a set of nonnegative integers $(o_{i,j})_{i \neq j}$ such that $o_{i,j} + o_{j,i} = 1$. Define
\[
A_{i,k}= Y_{i,k}Y_{i,k+1} \prod_{j \neq i}Y_{j,k+o_{j,i}}^{C_{ji}}.
\]
Then the Kashiwara operators can be defined as
\begin{equation}\label{eq:infop}
\e_i M = \begin{cases} 0 & \text{ if } \varepsilon_i(M) = 0, \\
A_{i,k_e}M & \text{ if } \varepsilon_i(M) >0, \end{cases} 
\ \ \ \ \ \ \ \ \
\f_i M = A_{i,k_f}^{-1} M.
\end{equation}
In finite types, the set $\hM$ together with the maps $\wt, \varepsilon_i, \varphi_i, \e_i, \f_i \ ( i \in I)$ forms an abstract $U_q(\g)$-crystal \cite{KKS:07}.  However, in affine types, $\hM$ is only an abstract $U_q'(\g)$-crystal.  

\begin{remark}
In this paper, we will work only in types $A_n^{(1)}$ $(n\ge 1)$ and $B_n^{(1)}$.  Once and for all, we fix our choice of integers $(o_{i,j})_{i\neq j}$ for each type: set
\[
o_{i,j} = \begin{cases} 1 & i < j, \\ 0 & i > j. \end{cases}
\]
Note, however, that $I$ is identified with $\ZZ/(n+1)\ZZ$ in type $A_n^{(1)}$, so in this case we also assert $o_{0,n} = 0$ and $o_{n,0} = 1$.  This is the convention used in \cite{KKS:07}.
\end{remark}

Define $\MM(\infty)$ to be the connected component of $\hM$ (under the application of the Kashiwara operators) containing $\one$.

\begin{thm}[\cite{KKS:07}]
The morphism $B(\infty) \longrightarrow \MM(\infty)$ given by $u_{\infty} \mapsto \one$ is a $U_q(\g)$-crystal isomorphism when $\g$ is of finite type and is a $U_q'(\g)$-crystal isomorphism when $\g$ is of affine type. 
\end{thm}

\begin{ex}\label{ex:A11-1}
Let $\MM(\infty)$ be of type $A_1^{(1)}$ and set $M = Y_{0,0}^{-1}Y_{0,1}Y_{1,1}Y_{1,2}^{-1}\one$.  Then, by the definition of the weight function on $\MM(\infty)$ from \cite{KKS:07} above we have $\wt(M) = 0$.  However, using the crystal axioms, since $M = \f_1\f_0\one$, the weight should be $\wt(M) = - \alpha_0 - \alpha_1 = -\delta$.  Note that if we use the expression of elements in $\MM(\infty)$ in terms of the variables $A_{i,k}$, then $M = A_{1,1}^{-1}A_{0,0}^{-1}\one$, from which it is clear an application of $\f_0$ has occurred in the path from $\one$ to $M$.
\end{ex}

\begin{remark}
In the last example, we used the facts that $\wt(\f_ib) = \wt(b) - \alpha_i$ for all $i \in I$ and that $\alpha_0 = \delta-\theta$, where $\theta$ is the highest root of the classical underlying root system of $\g$ (since $\g$ is assumed to be of untwisted affine type).  In particular,
\[
\wt(\f_0b) = \wt(b) - \alpha_0 = \wt(b) + \theta - \delta.
\]
\end{remark}

The goal of this paper is to obtain an expression for the coefficient of $\delta$ in the weight function solely in terms of the variables $Y_{i,k}$.

\begin{prob}\label{prob:main}
For $\MM(\infty)$ with $\g$ of affine type, determine a function $D\colon \MM(\infty) \longrightarrow \ZZ$ such that the weight function $\wt\colon \MM(\infty) \longrightarrow P$ is defined by
\[
\wt(M) = \sum_{i\in I} \Bigl( \sum_{k\ge 0} y_{i,k} \Bigr) \Lambda_i + D(M)\delta,
\]
where $M = \prod_{i\in I} \prod_{k\ge 0} Y_{i_k}^{y_{i,k}}$.
\end{prob}


Henceforth, assume $\g$ is of type $A_n^{(1)}$ or $B_n^{(1)}$.  Since each of $\theta$ and $\alpha_i$ $(i \in I \setminus \{0\})$ can be expressed as an element of $\ZZ \Lambda_0 \oplus \cdots \oplus \ZZ \Lambda_n$, it must be that, for $M = \f_{b_1} \f_{b_2} \cdots \f_{b_\ell} \one = \prod_{i\in I}\prod_{k \geq 0} Y_{i,k}^{y_{i,k}}\one \in \MM(\infty)$, 
\[
\wt(M) = \sum_{i \in I}\Big(\sum\limits_{k \geq 0} y_{i,k}\Big) \Lambda_i + \Bigl|\{1 \le j \le \ell : b_j = 0\}\Bigr| \delta.
\]
Note that this implies that, whenever $M$ can be expressed uniquely as $\prod_{i \in I}\prod_{k \geq 0} A_{i,k}^{a_{i,k}}\one$, the coefficient of $\delta$ in the weight of $M$ is exactly $\sum_{k \geq 0} a_{0,k}$.  Therefore, to complete the weight function for affine crystals, it suffices to calculate the number of $0$-arrows applied from $\one$ to reach $M$ in the crystal graph.

A solution to Problem \ref{prob:main} will be given in Section \ref{sec:typeA} for type $A_n^{(1)}$ and in Section \ref{sec:typeB} for type $B_3^{(1)}$.


\subsection{Generalized Young Walls}

Let $\mathcal B$ be a board with coloring as follows:
\[
\begin{tikzpicture}[baseline=50,font=\footnotesize,scale=.7]
 \draw (-6,0) to (-6,6.25);
 \draw (-5,0) to (-5,6.25);
 \draw (-4,0) to (-4,6.25);
 \draw (-3,0) to (-3,6.25);
 \draw (-2,0) to (-2,6.25);
 \draw (-1,0) to (-1,6.25);
 \draw[very thick] (-7,0) to ( 0,0) to ( 0,6.25);
 \draw (-7,1) to (0,1);
 \draw (-7,2) to (0,2);
 \draw (-7,3) to (0,3);
 \draw (-7,4) to (0,4);
 \draw (-7,5) to (0,5);
 \draw (-7,6) to (0,6);
 \draw[fill=white,color=white] (-6.5,2.25) rectangle (.5,2.75);
 \node at (-.5,5.5) {$1$};
 \node at (-.5,4.5) {$0$};
 \node at (-.5,3.5) {$n$};
 \node at (-.5,2.65) {$\vdots$};
 \node at (-.5,1.5) {$1$};
 \node at (-.5,0.5) {$0$};
 \node at (-1.5,5.5) {$0$};
 \node at (-1.5,4.5) {$n$};
 \node at (-1.5,3.5) {\tiny $n$$-$$1$};
 \node at (-1.5,2.65) {$\vdots$};
 \node at (-1.5,1.5) {$0$};
 \node at (-1.5,0.5) {$n$};
 \node at (-2.5,0.5) {$\cdots$};
 \node at (-2.5,1.5) {$\cdots$};
 \node at (-2.5,3.5) {$\cdots$};
 \node at (-2.5,4.5) {$\cdots$};
 \node at (-2.5,5.5) {$\cdots$};
 \node at (-3.5,5.5) {$2$};
 \node at (-3.5,4.5) {$1$};
 \node at (-3.5,3.5) {$0$};
 \node at (-3.5,2.65) {$\vdots$};
 \node at (-3.5,1.5) {$2$};
 \node at (-3.5,0.5) {$1$};
 \node at (-4.5,5.5) {$1$};
 \node at (-4.5,4.5) {$0$};
 \node at (-4.5,3.5) {$n$};
 \node at (-4.5,2.65) {$\vdots$};
 \node at (-4.5,1.5) {$1$};
 \node at (-4.5,0.5) {$0$};
 \node at (-5.5,5.5) {$0$};
 \node at (-5.5,4.5) {$n$};
 \node at (-5.5,3.5) {\tiny $n$$-$$1$};
 \node at (-5.5,2.65) {$\vdots$};
 \node at (-5.5,1.5) {$0$};
 \node at (-5.5,0.5) {$n$};
 \node at (-6.5,0.5) {$\cdots$};
 \node at (-6.5,1.5) {$\cdots$};
 \node at (-6.5,3.5) {$\cdots$};
 \node at (-6.5,4.5) {$\cdots$};
 \node at (-6.5,5.5) {$\cdots$};
 \node at (-.5,6.65) {$\vdots$};
 \node at (-1.5,6.65) {$\vdots$};
 \node at (-3.5,6.65) {$\vdots$};
 \node at (-4.5,6.65) {$\vdots$};
 \node at (-5.5,6.65) {$\vdots$};
\end{tikzpicture}\ .
\]

\begin{dfn}
The \defn{generalized Young walls} are constructed by placing $i$-colored boxes ($i \in I$) on the board ${\mathcal B}$ subject to the conditions:
\begin{enumerate}
\item the boxes are colored according to the board;
\item the colored boxes are placed in rows starting from the right.
\end{enumerate}
\end{dfn}

\begin{dfn}
A generalized Young wall is said to be \defn{proper} if, for each $p>q$ such that $p-q \equiv 0 \bmod n+1$, the number of boxes in the $p$th row from the bottom is less than or equal to the number of boxes in the $q$th row from the bottom.
\end{dfn}

\begin{ex}
Consider the following arrangement of boxes on the board $\mathcal B$ for $n=3$:
\[
Y=\wall{{0,3,2,1,0},{1},{2,1},{},{0,3,2}}\ , \ \ \ \ \ Y'=\wall{{0,3},{1},{},{},{0,3,2,1}}\ , \ \ \ \ \
Y'' = 
\begin{tikzpicture}[baseline=10,scale=.45]
\GYW{{0,,2},{1}}
\draw[very thick,white,fill=white] (-0.6,0) rectangle (-1.4,1.5);
\end{tikzpicture}\ .
\]
Then $Y$ and $Y'$ are both generalized Young walls, but $Y''$ is not since there is a break in the first row. Furthermore, $Y$ is proper.  The wall $Y'$ is not proper because the fifth row has four elements but the first row has two elements, and $5-1 \equiv 0 \bmod 4$.
\end{ex}

\begin{dfn}
The $k$th column $Y_k$ (from the right) of a generalized Young wall, for $k\ge 1$, contains a \defn{removable $\delta$} if one of each $i$ colored box can be removed from $Y_k$ and still obtain a generalized Young wall. In other words, if $a_{i,k}$ is the number of $i$-colored boxes in the $k$th column $Y_k$ ($i\in I$, $k\ge 1$), then $Y_k$ contains a removable $\delta$ if
\[
a_{i-1, k+1} < a_{i,k} \qquad \text{ for all } i \in I.
\]
If a generalized Young wall contains no removable $\delta$, it is said to be \defn{reduced}.
\end{dfn}


\begin{ex}
Consider the following generalized Young walls for $n=2$.
\begin{align*}
Y= \wall{{0,2,1,0},{1,0,2},{2,1,0}} \ , && Y' = \wall{{0,2,1,0},{1,0,2},{2,1,0},{0,2,1}}\ .
\end{align*}
Then $Y$ is reduced since removing a $0,1,$ and $2$ from any given column would leave the $0$ in the fourth column separated from its row, and therefore there is no removable $\delta$.
On the other hand, $Y'$ has a removable $\delta$ in the third column, so is not reduced.
\end{ex}

Let $\mathcal F(\infty)$ denote the set of all proper generalized Young walls. Let $\mathcal Y(\infty)$ denote the set of all proper reduced generalized Young walls. Note that $\mathcal Y(\infty) \subset \mathcal F(\infty)$.

Given any $Y \in \mathcal F(\infty)$, say that the leftmost box of any row is \defn{removable} and, if it is $i$-colored, then it is called a \defn{removable $i$-box}. Also, define the site left of the leftmost box in each row to be \defn{admissible}, and if a row has no boxes, then its rightmost site is admissible. If the site is $i$-colored, then it is called an \defn{$i$-admissible slot}. 

For any $Y \in \mathcal F(\infty)$, let $y_1, y_2, \ldots$ be the removable $i$-boxes and $i$-admissible slots ordered from left to right and bottom to top. The $i$-signature of $y_j$ is said to be $-$ if $y_j$ is removable and $+$ if $y_j$ is admissible. Then the $i$-signature of $Y$ is obtained by producing the sequence of $i$-signatures of $y_1, y_2, \ldots$ and then canceling out any $(+,-)$ pairs, resulting in a sequence of $-$'s followed by $+$'s. 

Define $\f_i Y$ to be the proper generalized Young wall obtained by placing an $i$-colored box at the site corresponding to the leftmost $+$ in the $i$-signature of $Y$ and $\e_i Y$ to be the proper generalized Young wall obtained by removing the $i$-box corresponding to the rightmost $-$ in the $i$-signature of $Y$. If no such $-$ exists, define $\e_i Y = 0$. Also, define the maps
\begin{align*}
\wt(Y) &= -\sum_{i \in I} k_i \alpha_i,\\
\varepsilon_i(Y) &= \text{ the number of } -\text{'s in the } i\text{-signature of } Y,\\
\varphi_i(Y) &= \varepsilon_i(Y) + \langle h_i, \wt(Y)\rangle.
\end{align*}
Here, $k_i$ is the number of $i$-colored boxes in $Y$ and the $\alpha_i$ are as defined for $U_q(A_n^{(1)})$. Then $\mathcal F(\infty)$ together with the maps above form an abstract $U_q(A_n^{(1)})$-crystal. For $\mathcal Y(\infty)$, more can be said. 

\begin{thm}[\cite{KS:10}]
The morphism $B(\infty) \longrightarrow \mathcal Y(\infty)$ such that $u_\infty \mapsto \varnothing$, where $\varnothing$ is the empty generalized Young wall, defines a $U_q(A_n^{(1)})$-crystal isomorphism.
\end{thm}

There is a map $\Psi\colon \mathcal{F}(\infty) \longrightarrow \hM$ defined by
\begin{equation}\label{FtoM}
\Psi(Y) = \prod_{i \in I}\prod_{k \geq 0} A_{i,k}^{-a_{i,k+1}}\one,
\end{equation}
where $a_{i,k}$ is the number of $i$-colored boxes in the $k$th column of $Y$.
It can be seen using results from \cite{KS:10} that $\mathcal Y(\infty)$ is isomorphic to $\MM(\infty)$ as a $U_q(A_n^{(1)})$-crystal using the restriction of \eqref{FtoM}.

\section{The weight function in type $A_n^{(1)}$}\label{sec:typeA}

\subsection{Specialization to type $A_1^{(1)}$}

In type $A_1^{(1)}$, we have
\begin{align*}
A_{0,k} &= Y_{0,k}Y_{0,k+1}Y_{1,k+1}^{-2}, &
A_{1,k} &= Y_{1,k}Y_{1,k+1}Y_{0,k}^{-2}.
\end{align*}

\begin{lemma}\label{thm:1}
For $M = \prod_{k \geq 0} Y_{0,k}^{y_{0,k}}Y_{1,k}^{y_{1,k}}\one \in \MM(\infty)$, define $a_{i,k}$ recursively as follows:
\begin{equation}\label{eq:A11recursion}
\begin{aligned}
a_{1,0} &= y_{1,0}, &
a_{1,k} &= y_{1,k} + 2a_{0,k-1}-a_{1,k-1},\\
a_{0,0} &= y_{0,0} + 2a_{1,0}, &
a_{0,k} &= y_{0,k} + 2a_{1,k} - a_{0,k-1}.
\end{aligned}
\end{equation}
Then $M = \prod_{k \geq 0} A_{0,k}^{a_{0,k}} A_{1,k}^{a_{1,k}}\one$.
\end{lemma}

\begin{proof}
Since $M \in \MM(\infty)$, there exists some $a_{0,k}, a_{1,k} \in \ZZ$ such that $M = \prod_{k \geq 0} A_{0,k}^{a_{0,k}} A_{1,k}^{a_{1,k}}\one$. Thus, it suffices to show that the recurrence holds for these values $a_{i,k}$.
Expanding the terms $A_{0,k}^{a_{0,k}}$ and $A_{1,k}^{a_{1,k}}$ shows that
\[
M = Y_{0,0}^{a_{0,0}-2a_{1,0}} Y_{1,0}^{a_{1,0}}Y_{0,1}^{a_{0,1}}Y_{1,1}^{a_{1,1}-2a_{0,1}} \prod_{k \geq 1} A_{0,k}^{a_{0,k}} A_{1,k}^{a_{1,k}}\one.
\]
Since none of the terms in $\prod_{k \geq 1} A_{0,k}^{a_{0,k}} A_{1,k}^{a_{1,k}}\one$ contribute to the power on $Y_{0,0}$ or $Y_{1,0}$, equating powers yields
$y_{1,0} = a_{1,0}$
and
$y_{0,0} = a_{0,0} - 2a_{1,0}$, so
$a_{0,0} = y_{0,0} + 2 a_{1,0}$.
Thus, the first two desired equations hold.

Next, for some $m \geq 1$, consider
\[
M = \left(\prod_{k=0}^{m-2} A_{0,k}^{a_{0,k}}A_{1,k}^{a_{1,k}}\right)A_{0,m-1}^{a_{0,m-1}}A_{1,m-1}^{a_{1,m-1}}A_{0,m}^{a_{0,m}}A_{1,m}^{a_{1,m}} \left(\prod_{k\ge m+1} A_{0,k}^{a_{0,k}}A_{1,k}^{a_{1,k}}\right)\one.
\]
Note that $\prod_{k=0}^{m-2} A_{0,k}^{a_{0,k}}A_{1,k}^{a_{1,k}}\one$, when expanded, only yields values of $Y_{i,k}$ where $k < m$, and the product $\prod_{k\ge m+1} A_{0,k}^{a_{0,k}}A_{1,k}^{a_{1,k}}\one$ will only yield powers of $Y_{i,k}$ where $k > m$. Thus, in the expansion of $A_{1,m-1}^{a_{1,m-1}}A_{0,m}^{a_{0,m}}A_{1,m}^{a_{1,m}}$, the powers on $Y_{0,m}$ and $Y_{1,m}$ will be exactly $y_{0,m}$ and $y_{1,m}$ respectively. Since
\[
A_{1,m-1}^{a_{1,m-1}}A_{0,m}^{a_{0,m}}A_{1,m}^{a_{1,m}} = Y_{0,m-1}^{a_{0,m-1} -2a_{1,m-1}} Y_{1,m-1}^{a_{1,m-1}} Y_{0,m}^{a_{0,m-1}-2a_{1,m} + a_{0,m}} Y_{1,m}^{a_{1,m-1} -2a_{0,m-1} + a_{1,m}},
\]
equating powers gives
$y_{0,m} = a_{0,m-1}-2a_{1,m} + a_{0,m}$ and
$y_{1,m} = a_{1,m-1} -2a_{0,m-1} + a_{1,m}$.
These can be rearranged to yield the remaining desired equations.
\end{proof}

For the following lemma, use the convention that an empty sum is $0$.

\begin{lemma}\label{thm:2}
Let $M = \prod_{k \ge0 } Y_{0,k}^{y_{0,k}}Y_{1,k}^{y_{1,k}}\one = \prod_{k \geq 0} A_{0,k}^{a_{0,k}} A_{1,k}^{a_{1,k}}\one \in \MM(\infty)$. Then, for all $k \geq 0$, we have
\begin{align*}
a_{0,k} &= \sum_{i=0}^k (2i+1) y_{0,k-i} + (2i+2) y_{1,k-i},\\
a_{1,k} &= (2k+1)y_{1,0} + \sum_{i=0}^{k-1} (2i+1)y_{1,k-i} + (2i+2)y_{0,k-i-1}.
\end{align*}
\end{lemma}

\begin{proof}
Note that \eqref{eq:A11recursion} can be rewritten as a single recurrence relation.  Namely, for $i \geq 1$, define $z_{2i-1} = a_{1,i-1}$, $z_{2i} = a_{0,i-1}$, $r_{2i-1}= y_{1,i-1}$, and $r_{2i} = y_{0,i-1}$ with the additional condition that $r_0=z_0=0$. Then the Lemma \ref{thm:1} equations can be encoded as
\[
z_0 = 0, \ \ \
z_1= r_1, \ \ \
z_k = r_k + 2z_{k-1}-z_{k-2}.
\]
Let $g(x) = \sum_{k\ge0} z_k x^k$ be the generating function for $(z_k)_{k\ge0}$. Then
\begin{align*}
g(x) &= r_1 x + \sum_{k\ge2} z_kx^k \\
&= r_1 x + \sum_{k\ge2} (r_k + 2z_{k-1}-z_{k-2})x^k \\
&= r_1 x + \sum_{k\ge2} r_k + 2x \sum_{k\ge2} z_{k-1}x^{k-1} - x^2\sum_{k\ge2} z_{k-2}x^{k-2}.
\end{align*}
Note that the summations $\sum_{k\ge2} z_{k-1}x^{k-1}$ and $\sum_{k\ge2} z_{k-2}x^{k-2}$ can both be reindexed to show equivalence to $g(x)$. Thus,
\[
g(x) = r_1 x + \sum_{k\ge2} r_k + 2x g(x) -x^2 g(x).
\]
Solving the above equation for $g(x)$ yields
\[
g(x) = \frac{\sum_{k\ge0} r_k }{x^2 -2x + 1}
= \frac{\sum_{k\ge0} r_k}{(x-1)^2}
= \frac{1}{(1-x)^2} \sum_{k\ge0} r_k
= \biggl(\sum_{i\ge0} (i+1)x^i\biggr) \biggl(\sum_{k\ge0} r_k\biggr).
\]
This can be expanded to a single power series to yield
\[
g(x) = \sum_{k\ge0} \left(\sum_{i=0}^k (i+1)r_{k-i}\right) x^k.
\]
Thus,
\[
z_k = \sum_{i=0}^k (i+1)r_{k-i}.
\]

This can be broken up into two different cases to show the desired result. First, note that
\[
a_{0,k} = z_{2k+2} = \sum_{i=0}^{2k+2} (i+1) r_{2k+2-i}.
\]
The sum can be broken up into even and odd parts, so
\begin{align*}
z_{2k+1} &= \sum_{i=0}^{k+1} (2i+1) r_{2k-2i+2} + \sum_{i=0}^{k} (2i+2)r_{2k-2i+1} \\
&= (2k+3) r_0 + \sum_{i=0}^{k} \bigl( (2i+1) r_{2k-2i+2} + (2i+2) r_{2k-2i+1} \bigr) \\
&= \sum_{i=0}^{k} \bigl( (2i+1) y_{0,k-i} + (2i+2) y_{1,k-i} \bigr).
\end{align*}
Now, to prove the final part of the theorem, consider
\begin{align*}
a_{1,k} = z_{2k+1} &= \sum_{i=0}^{2k+1}(i+1) r_{2k+1-i}\\
&= \sum_{i=0}^{k} (2i+1) r_{2k-2i+1} + \sum_{i=0}^{k} (2i+2) r_{2k-2i}\\
&= (2k+2) r_0 + (2k+1) r_1 + \sum_{i=0}^{k-1} (2i+1) r_{2k-2i+1} + \sum_{i=0}^k (2i+2) r_{2k-2i}\\
&= (2k+1) y_{1,0} + \sum_{i=0}^{k-1} \bigl((2i+1)y_{1,k-i} + (2i+2) y_{0,k-i-1}\bigr). \qedhere
\end{align*}
\end{proof}

Note that Lemma \ref{thm:2} implies that the $a_{i,k}$ are uniquely determined for a given $M$.  This gives the following result.

\begin{thm}\label{prop:1}
If $M = \prod_{i \in I} \prod_{k \geq 0} Y_{i,k}^{y_{i,k}}\one \in \MM(\infty)$, then 
\[
D(M) = \sum_{k\ge0} \sum_{j=0}^{k} \bigl( (2j+1) y_{0,k-j} + (2j+2) y_{1,k-j} \bigr).
\]
\end{thm}

\begin{ex}
Let $M =\f_0 \f_1 \f_0 \f_0 \f_0 \one$. Then $M = Y_{0,0}^{-3}Y_{0,1}^{-2}Y_{0,2}^{-1}Y_{1,1}^5Y_{1,2}\one$ and
\begin{align*}
\sum_{k\ge0} \sum_{j=0}^k & \bigl( (2j+1) y_{0,k-j} + (2j+2) y_{1,k-j} \bigr) \\
&= (y_{0,0} + 2y_{1,0}) + (y_{0,1} + 2y_{1,1} + 3y_{0,0} + 4y_{1,0}) \\
& \ \ \ \ + (y_{0,2} + 2y_{1,2} + 3y_{0,1} + 4y_{1,1} + 5y_{0,0} + 6y_{1,0})  \\
& \ \ \ \ + (y_{0,3} + 2y_{1,3} + 3y_{0,2} + 4y_{1,2} + 5y_{0,1} + 6y_{1,1} + 7y_{0,0} + 8y_{1,0}) \\
& \ \ \ \ + (y_{0,4} + 2y_{1,4} + 3y_{0,3} + 4y_{1,3} + 5y_{0,2} + 6y_{1,2} + 7y_{0,1} + 8y_{1,1} + 9y_{0,0} + 10y_{1,0}) + \cdots \\
&= \bigl(-3+2(0)\bigr) + \bigl(-2 + 2(5) + 3(-3) + 4(0)\bigr) \\
& \ \ \ \ + \bigl(-1 + 2(1) + 3(-2) + 4(5) + 5(-3) + 6(0)\bigr) \\
& \ \ \ \ + \bigl(0 + 2(0) + 3(-1) + 4(1) + 5(-2) + 6(5) + 7(-3) + 8(0)\bigr) \\
& \ \ \ \ + \bigl(0 + 2(0) + 3(0) + 4(0) + 5(-1) + 6(1) + 7(-2) + 8(5) + 9(-3) + 10(0)\bigr) +\cdots \\
&= -3 - 1 + 0 + 0 + 0 + \cdots.
\end{align*}
Thus, by Theorem \ref{prop:1}, we have $D(M) = -4$. This matches the number of $\f_0$'s applied to reach $M$, which is the expected result.
\end{ex}

\begin{ex}
Let $M=\f_0 \f_1 \f_1 \f_0 \one = Y_{0,0}^{-1} Y_{0,1}^2 Y_{0,2}^{-1}\one$. Applying Lemma \ref{thm:2} gives
\begin{align*}
a_{0,0} &= y_{0,0}+2y_{1,0}= -1,\\
a_{0,1} &= y_{0,1} + 2y_{1,1} + 3y_{0,0} + 4y_{1,0} = 2-3=-1,\\
a_{0,2} &= y_{0,2} + 2y_{1,2} + 3y_{0,1} + 4y_{1,1}+5y_{0,0} +6y_{1,0}=-1+6-5=0.
\end{align*}
Adding up these values shows that $D(M) = -2$ according to Theorem \ref{prop:1}, which coincides with the fact that $\wt(M) = \wt(\f_0 \f_1 \f_1 \f_0 \one) = -2\alpha_0 - 2\alpha_1 = -2\delta$.
\end{ex}

\subsection{General result for type $A_n^{(1)}$, $n\ge 2$}

In the case $A_n^{(1)}$ for $n\ge 2$, the same method used for $A_1^{(1)}$ does not yield an explicit formula for the weight function. Consider the following result.

\begin{lemma}\label{thm:3}
If $M = \prod_{k \geq 0} \prod_{i \in I} Y_{i,k}^{y_{i,k}}\one = \prod_{k \geq 0}\prod_{i\in I} A_{i,k}^{a_{i,k}}\one \in \MM(\infty)$, then
\begin{equation}\label{eq:An1recursion}
a_{i,k} -a_{i-1,k} = \sum_{\ell=0}^k y_{i+\ell, k-\ell}
\end{equation}
for all $i \in I$ and $k \geq0$.
\end{lemma}

\begin{proof}

This will be a proof by induction on $k$. As a base case, note that when $k=0$, 
\[
a_{i,0}-a_{i-1,0} = y_{i,0}.
\]
To show that this holds, note that the only terms in  $\prod_{k \geq 0}\prod_{i\in I} A_{i,k}^{a_{i,k}}\one$ that contain any $Y_{i,0}$ are $A_{i,0}$ and $A_{i-1,0}$. These contribute $Y_{i,0}$ and $Y_{i,0}^{-1}$ respectively to the overall product, so certainly $y_{i,0}=a_{i,0}-a_{i-1,0}.$

Next, for the sake of induction, assume that the result holds for all values of $i$ and a given value $k$. Then note that
\[
\sum_{\ell=0}^{k+1} y_{i+\ell,k+1-\ell} = y_{i,k+1} + \sum_{\ell=1}^{k+1} y_{i+\ell,k+1-\ell}=
y_{i,k+1} + \sum_{\ell=0}^k y_{i+\ell+1,k-\ell}.
\]
By the inductive assumption, this is equal to
\[
y_{i,k+1} + a_{i+1,k} - a_{i,k}.
\]
Now consider $y_{i,k+1}$. Note that $y_{i,k+1} = a_{i,k} - a_{i+1, k} + a_{i,k+1} - a_{i-1,k+1}$. This equation can be obtained by considering the values of $\prod_{m \geq 0}\prod_{j\in I} A_{j,m}^{a_{j,m}}$ that contribute to $Y_{i,k}$. Using this equation yields
\begin{align*}
\sum_{m=0}^{k+1} y_{i+m, k+1-m} &=a_{i,k} - a_{i+1, k} + a_{i,k+1} - a_{i-1,k+1}+a_{i+1,k}-a_{i,k}\\
&= a_{i,k+1}-a_{i-1,k+1},
\end{align*}
so the result holds by induction.
\end{proof}

Consider the general solution to the above equations for a fixed $k$. Note that if a specific solution is given by integers $r_i$ such that $a_{i,k} = r_i$ for $i \in I$, then for any $t \in \ZZ$, another solution is given by $a_{i,k} = r_i + t$ for all $i \in I$. Consider now the following motivating example.

\begin{ex}
In type $A_2^{(1)}$, let $M =\f_1 \f_0 \one = Y_{0,0}^{-1}Y_{1,1}^{-1}Y_{2,0}Y_{2,1}\one \in \MM(\infty)$.  This can be written as $A_{0,0}^{-2}A_{1,0}^{-2}A_{2,0}^{-1}A_{0,1}^{-1}A_{1,1}^{-1}A_{2,1}^{-1}\one$. Consider the following generalized Young wall
\[
Y'' = \wall{{0,2},{1,0},{2,1},{0},{1}} \in \mathcal F(\infty).
\]
Then $\Psi(Y'') = M$ (see Equation \eqref{FtoM}).
Note that the second column from the right has a removable $\delta$. Removing it yields the generalized Young wall
\[
Y' = \wall{{0},{1},{2},{0},{1}} \in \mathcal{F}(\infty) .
\]
This new generalized Young wall also has a removable $\delta$, this time in the first column. Removing it yields
\[
Y = \begin{tikzpicture}[baseline=5,scale=.45]
 \GYW{{0},{1}}
\end{tikzpicture} \in \mathcal{Y}(\infty) .
\]
Then $\Psi(Y) = A_{0,0}^{-1}A_{1,0}^{-1}\one$, which is in $\MM(\infty)$ because $\Psi$ defines the isomorphism between $\MM(\infty)$ and $\mathcal{Y}(\infty)$.  Expanding $\Psi(W)$ in terms of the $Y$-variables gives $Y_{0,0}^{-1}Y_{1,1}^{-1}Y_{2,0}Y_{2,1}\one$, and we claim that $Y$ is the reduced, proper generalized Young wall that corresponds to the original $M$.

To explain our claim, note that eliminating a removable $\delta$ from the $k$th column of a generalized Young wall is the same as subtracting $1$ from $a_{i,k}$ for each $i\in I$. It can be easily checked that, in type $A_n^{(1)}$,  $\prod_{i \in I} A_{i,k}\one = \one$ for any $k \in \ZZ_{\geq 0}$. Thus, eliminating a removable $\delta$ does not change which monomial a generalized Young wall corresponds. Therefore, the weight function on $\mathcal Y(\infty)$ can be used here to calculate the coefficient of $\delta$ on the weight function of elements of $\MM(\infty)$.
\end{ex}

This idea gives rise to the following algorithm for computing $D(M)$ for $M \in \MM(\infty)$.

\begin{thm}\label{thm:4}
Let $M = \prod_{k \geq0}\prod_{i \in I} Y_{i,k}^{y_{i,k}}\one \in \MM(\infty)$. Then $D(M)$ is computed using the following algorithm.
\begin{enumerate}
\item\label{alg:step1} Find the maximum value $m$ such that there exists an $i\in I$ with $y_{i,m+1} \neq 0$.
\item\label{alg:step2} Find the unique solution to $\{a_{i,m}\}_{i \in I}$ such that $a_{i,m} \in \ZZ_{\leq 0}$, for all $i\in I$, and such that there exists an $i\in I$ with $a_{i,m}=0$. Let these be the values for $a_{i,m}$. Set $k=m-1$.
\item\label{alg:step3} Find the maximal solution to $\{a_{i,k}\}_{i \in I}$ such that $a_{i,k} \in \ZZ_{\leq 0}$ and $a_{i,k} \leq a_{i-1,k+1}$ for all $i$. Decrease $k$ by $1$.
\item\label{alg:step4} Repeat step (\ref{alg:step3}) until $k =-1$.
\end{enumerate}
After finding the values of $a_{i,k}$ in this manner, we have $D(M) = \sum_{k=0}^m a_{0,k}$.
\end{thm}

\begin{proof}
First note that in step (\ref{alg:step2}), such an integral solution to the equations of \eqref{eq:An1recursion} exists since $M \in \MM(\infty)$ so must be expressible via the $A_{i,k}$. Furthermore, given a solution $(c_i)_{i \in I}$ of integers to the system of equations, any other solution $(c_i')_{i \in I}$ of integers is given by some integral shift of the $c_i$. This can be seen by first noting that there exists an integer $t$ such that $c_0' = c_0 + t$. Then note that 
\begin{align*}
c_0 - c_n &= c_0' - c_n'\\
c_0- c_n &= (c_0+t)-c_n'\\
-c_n &= t-c_n' \\
c_n' &= c_n + t.
\end{align*}
This can be repeated to show that $c_i' = c_i + t$ for all $i \in I$. Thus, the solution to step (\ref{alg:step2}) is the integral shift of the first solution such that some value is $0$ and the rest are nonpositive.

Showing the existence and uniqueness of the solution in step (\ref{alg:step3}) is nearly identical to the argument used for step (\ref{alg:step2}). It now suffices to show that the solution obtained by this algorithm represents the correct expression for $M$ corresponding to a proper, reduced generalized Young wall.

Given any nonpositive solution to the $a_{i,m}$, if none of the values are $0$, then the corresponding generalized Young wall has a removable $\delta$. Thus, the solution here should certainly be the one given by the algorithm.

For any $a_{i,k-1}$ such that $a_{i,k-1} > a_{i-1,k}$, the corresponding generalized Young wall will not be proper. Therefore, it is certainly true that $a_{i,k-1} \leq a_{i-1,k}$. Furthermore, if $a_{i,k-1}< a_{i-1,k}$ for all $i\in I$, then the corresponding generalized Young wall has a removable $\delta$ in its $(k-1)$st row and so is not reduced. Thus, it must be the maximal solution, where for at least one $i$, $a_{i,k-1}=a_{i-1,k}$.

Since the $a_{0,k}$ values correspond to the $0$-boxes in the generalized Young wall, we have $D(M) = \sum_{k=0}^m a_{0,k}$.
\end{proof}

\begin{ex}
Note that in $\MM(\infty)$ of type $A_4^{(1)}$, 
\[
\f_0\f_1\f_3\f_0\f_4\f_0\one = A_{0,0}^{-3} A_{1,0}^{-1} A_{3,2}^{-1} A_{4,1}^{-1}\one = Y_{0,0}^{-3} Y_{0,1}^{-1} Y_{1,0}^2 Y_{1,1}^{-1} Y_{2,0} Y_{2,3} Y_{3,3}^{-1} Y_{4,1}^2\one.
\]
Applying Theorem \ref{thm:4}, note that the desired value $m$ is 2. Step \ref{alg:step2} of the algorithm says that the correct values of $a_{i,2} \ (i \in I)$ are solutions to the system
\begin{align*}
a_{0,2}-a_{4,2} &= y_{0,2} + y_{1,1} + y_{2,0} = 0,\\
a_{1,2}-a_{0,2} &= y_{1,2} + y_{2,1} + y_{3,0} = 0,\\
a_{2,2}-a_{1,2} &= y_{2,2} + y_{3,1} + y_{4,0}=0,\\
a_{3,2}-a_{2,2} &= y_{3,2} + y_{4,1} + y_{0,0}=-1,\\
a_{4,2}-a_{3,2} &= y_{4,2} + y_{0,1} + y_{1,0}=1.
\end{align*}
The general solution to this system is (for any $t \in \ZZ$) $a_{0,2}= t$, $a_{1,2} = t$, $a_{2,2} = t$, $a_{3,2} = t-1$, and $a_{4,2} = t$. Note that the maximum solution to this system with nonpositive values is $a_{3,2} = -1$ and $a_{i,2} = 0$ for all other $i$.

Similarly, for $k=1$, the values of $a_{i,1} \ (i \in I)$ are a solution to the system
\begin{align*}
a_{0,1} - a_{4,1} &= y_{0,1} + y_{1,0}=1,\\
a_{1,1} -a_{0,1} &= y_{1,1} + y_{2,0}=0,\\
a_{2,1} - a_{1,1} &= y_{2,1} + y_{3,0}=0,\\
a_{3,1} - a_{2,1} &= y_{3,1} + y_{4,0}=0,\\
a_{4,1} - a_{3,1} &= y_{4,1} + y_{0,0}=-1.
\end{align*}
The maximal solution to this system such that $a_{4,1} \leq -1$ and $a_{i,1} \leq 0$ for all other $i$ is $a_{4,1} = -1$ and $a_{i,1} = 0$ for all other $i$. 

Finally, for $k=0$, the values for $a_{i,0} \ (i \in I)$ are solutions to the system
\begin{align*}
a_{0,0}-a_{4,0} &= y_{0,0}=-3,\\
a_{1,0}-a_{0,0} &= y_{1,0}=2,\\
a_{2,0}-a_{1,0} &= y_{2,0}=1,\\
a_{3,0}-a_{2,0} &= y_{3,0}=0,\\
a_{4,0}-a_{3,0} &= y_{4,0}=0.
\end{align*}
Note that the maximal solution to this such that $a_{0,0} \leq -1$ and $a_{i,0} \leq 0$ for all other $i$ is $a_{0,0} = -3$, $a_{1,0} = -1$, $a_{2,0} = 0$, $a_{3,0} = 0$, and $a_{4,0} = 0$. 
\end{ex}

\section{The weight function in type $B_3^{(1)}$}\label{sec:typeB}

\subsection{The result for $B_3^{(1)}$}
We now consider type $B_3^{(1)}$.

\begin{lemma}\label{lemma:1}
For $M = \prod_{i \in I} \prod_{k \geq 0} Y_{i,k}^{y_{i,k}} \one = \prod_{i \in I} \prod _{k \geq 0} A_{i,k}^{a_{i,k}}\one \in \MM(\infty)$, each of the following hold for $k \geq 0$ (with the convention that $a_{i,-1}=0$ for all $i \in I$):
\begin{equation}\label{eq:lem1}
\begin{aligned}
a_{0,k} &= y_{0,k} + a_{2,k-1} - a_{0,k-1},\\
a_{1,k} &= y_{1,k} + a_{2,k-1} - a_{1,k-1},\\
a_{2,k} &= y_{2,k}  + a_{0,k} + a_{1,k} + a_{3,k-1}-a_{2,k-1},\\
a_{3,k} &= y_{3,k} + 2a_{2,k} - a_{3,k-1}.
\end{aligned}
\end{equation}
\end{lemma}

\begin{proof} Consider which $A_{i,k}$ contribute $Y_{0,\ell}$ to $M$. They are exactly $A_{0,\ell}$, $A_{0,\ell-1}$, and $A_{2,\ell-1}$. Expanding these and equating powers with those of $Y_{0,\ell}$ gives
\[
y_{0,\ell} = a_{0,\ell} + a_{0,\ell-1} - a_{2,\ell-1},
\]
which gives the first desired equality. The other three follow similarly.
\end{proof}

The following lemma is straightforward, but we include it for easy reference in the arguments below.

\begin{lemma}\label{lemma:2}
For any $m \geq 0$, we have
$
\left\lfloor\frac{m}{2}\right\rfloor + \left\lfloor\frac{m+1}{2}\right\rfloor = m.
$
\end{lemma}

\begin{proof}
Note that $\left \lfloor \frac{m}{2} \right \rfloor$ is the number of positive even integers less than or equal to $m$. Similarly, $\left \lfloor \frac{m+1}{2} \right \rfloor$ is the number of positive odd integers less than or equal to $m$. Summing these values gives the number of integers less than or equal to $m$, which is exactly $m$. 
\end{proof}

\begin{lemma}\label{thm:5}
Given $M = \prod_{i \in I} \prod_{m \geq 0} Y_{i,m}^{y_{i,m}}\one \in \MM(\infty)$,  the solution to  $M =\prod_{i \in I} \prod _{m \geq 0} A_{i,m}^{a_{i,m}}\one$ is
\begin{align*}
a_{0,m} &= \sum_{k=0}^m \left(2\left\lfloor\frac{k}{2}\right\rfloor - \left\lfloor\frac{k-1}{2}\right\rfloor\right) y_{0,m-k} + \left\lfloor\frac{k+1}{2}\right\rfloor y_{1,m-k} + k y_{2,m-k} + \left\lfloor\frac{k}{2}\right\rfloor y_{3,m-k},\\
a_{1,m} &= \sum_{k=0}^m  \left\lfloor\frac{k+1}{2}\right\rfloor y_{0,m-k} + \left(2\left\lfloor\frac{k}{2}\right\rfloor - \left\lfloor\frac{k-1}{2}\right\rfloor\right) y_{1,m-k} + k y_{2,m-k} + \left\lfloor\frac{k}{2}\right\rfloor y_{3,m-k},\\
a_{2,m} &= \sum_{k=0}^m (k+1) y_{0,m-k} + (k+1) y_{1,m-k} + (2k+1) y_{2,m-k} + k y_{3,m-k},\\
a_{3,m} &= \sum_{k=0}^m 2\left\lfloor\frac{k+2}{2}\right\rfloor y_{0,m-k} + 2\left\lfloor\frac{k+2}{2}\right\rfloor y_{1,m-k} + (2k+2) y_{2,m-k} + \left(2\left\lfloor\frac{k}{2}\right\rfloor + 1\right) y_{3,m-k}.
\end{align*}
\end{lemma}

\begin{proof}
We proceed by induction on $m$. First, note that the base case of $a_{0,0} = y_{0,0}$, $a_{1,0} = y_{1,0}$, $a_{2,0} = y_{0,0} + y_{1,0} + y_{2,0}$ and $a_{3,0} = y_{3,0} + 2y_{0,0} + 2 y_{1,0} + 2y_{2,0}$ each hold by the Lemma~\ref{lemma:1}.

Now assume for the sake of induction that for a fixed $m \in \ZZ_{\geq 0}$ the result holds for $a_{i,m}$ for all $i \in \{0,1,2,3\}$. Then, by \eqref{eq:lem1}, we have
\begin{align*}
a_{0,m+1} =& \ y_{0,m+1} + \sum_{k=0}^m (k+1) y_{0,m-k} + (k+1) y_{1,m-k} + (2k+1) y_{2,m-k} + k y_{3,m-k}\\
&\, - \sum_{k=0}^m \left(2\left\lfloor\frac{k}{2}\right\rfloor - \left\lfloor\frac{k-1}{2}\right\rfloor\right) y_{0,m-k} + \left\lfloor\frac{k+1}{2}\right\rfloor y_{1,m-k} + k y_{2,m-k} + \left\lfloor\frac{k}{2}\right\rfloor y_{3,m-k}.
\end{align*}
Since the desired identity is (by evaluating the $k=0$ term and reindexing the sum)
\begin{align*}
    a_{0,m+1} =&\ y_{0,m+1} + \sum_{k=0}^{m}\left(2\left\lfloor\frac{k+1}{2}\right\rfloor - \left\lfloor\frac{k}{2}\right\rfloor\right)y_{0,m-k} \\ 
    &+\ \left\lfloor\frac{k+2}{2}\right\rfloor y_{1,m-k} + (k+1)y_{2,m-k} + \left\lfloor\frac{k+1}{2}\right\rfloor y_{2,m-k},
\end{align*}
it suffices to prove the following four identities:
\begin{subequations}
\begin{align}
2\left\lfloor\frac{k+1}{2}\right\rfloor - \left\lfloor\frac{k}{2}\right\rfloor &=k+1-2\left\lfloor\frac{k}{2}\right\rfloor+\left\lfloor\frac{k-1}{2}\right\rfloor, \label{eq:pf43a}\\
\left\lfloor\frac{k+2}{2}\right\rfloor &= k+1-\left\lfloor\frac{k+1}{2}\right\rfloor,\label{eq:pf43b}\\
k+1 &= 2k+1-k,\label{eq:pf43c}\\
\left\lfloor\frac{k+1}{2}\right\rfloor &= k - \left\lfloor\frac{k}{2}\right\rfloor.\label{eq:pf43d}
\end{align}
\end{subequations}
The relation \eqref{eq:pf43c} holds trivially and both \eqref{eq:pf43b} and \eqref{eq:pf43d} are each immediate results from Lemma \ref{lemma:2}. To prove \eqref{eq:pf43a}, note that
\begin{align*}
2\left\lfloor\frac{k+1}{2}\right\rfloor - \left\lfloor\frac{k}{2}\right\rfloor + 2\left\lfloor\frac{k}{2}\right\rfloor - \left\lfloor\frac{k-1}{2}\right\rfloor &= 2\left(\left\lfloor \frac{k+1}{2}\right\rfloor + \left\lfloor \frac{k}{2}\right\rfloor\right) - \left(\floor{\frac{k}{2}} + \floor{\frac{k-1}{2}}\right)\\
&= 2k - (k-1)\\
&=k+1,
\end{align*}
which is equivalent to the first statement. Thus, the result holds for $a_{0,m+1}$.

Now consider similarly $a_{1,m+1}$. Using a method identical to what was used for $a_{0,m+1}$, it can be seen that this case holds if the following four identities hold:
\begin{subequations}
\begin{align}
\left\lfloor\frac{k+2}{2}\right\rfloor &= k+1-\left\lfloor\frac{k+1}{2}\right\rfloor, \label{eq:pf43a2}\\
2\left\lfloor\frac{k+1}{2}\right\rfloor - \left\lfloor\frac{k}{2}\right\rfloor &= k+1-2\left\lfloor\frac{k}{2}\right\rfloor+\left\lfloor\frac{k-1}{2}\right\rfloor,\label{eq:pf43b2}\\
k+1 &= 2k+1-k,\label{eq:pf43c2}\\
\left\lfloor\frac{k+1}{2}\right\rfloor &= k - \left\lfloor\frac{k}{2}\right\rfloor.\label{eq:pf43d2}
\end{align}
\end{subequations}
Since each of these are identical to an identity used to prove the result for $a_{0,m+1}$, the result holds for $a_{1,m+1}$ as well.

Next, to show that the theorem holds for $a_{2,m+1}$, note first that by Lemma~\ref{lemma:1}:
\begin{align*}
&a_{2,m+1} = y_{2,m+1} \\ &+ \sum_{k=0}^{m+1} \left(2\left\lfloor\frac{k}{2}\right\rfloor - \left\lfloor\frac{k-1}{2}\right\rfloor\right) y_{0,m+1-k} + \left\lfloor\frac{k+1}{2}\right\rfloor y_{1,m+1-k} + k y_{2,m+1-k} + \left\lfloor\frac{k}{2}\right\rfloor y_{3,m+1-k}\\
&+\sum_{k=0}^{m+1}  \left\lfloor\frac{k+1}{2}\right\rfloor y_{0,m+1-k} + \left(2\left\lfloor\frac{k}{2}\right\rfloor - \left\lfloor\frac{k-1}{2}\right\rfloor\right) y_{1,m+1-k} + k y_{2,m+1-k} + \left\lfloor\frac{k}{2}\right\rfloor y_{3,m+1-k}\\
&+\sum_{k=0}^m 2\left\lfloor\frac{k+2}{2}\right\rfloor y_{0,m-k} + 2\left\lfloor\frac{k+2}{2}\right\rfloor y_{1,m-k} + (2k+2) y_{2,m-k} + \left(2\left\lfloor\frac{k}{2}\right\rfloor + 1\right) y_{3,m-k}\\
&-\sum_{k=0}^m (k+1) y_{0,m-k} + (k+1) y_{1,m-k} + (2k+1) y_{2,m-k} + k y_{3,m-k}.
\end{align*}
Pulling out the $k=0$ term of the first two sums and reindexing so that each of the sums match gives
\begin{align*}
&a_{2,m+1} = y_{0,m+1}+y_{1,m+1}+y_{2,m+1} \\
&+\sum_{k=0}^m \left(2\floor{\frac{k+1}{2}} - \floor{\frac{k}{2}}\right)y_{0,m-k} + \floor{\frac{k+2}{2}}y_{1,m-k} + (k+1)y_{2,m-k} + \floor{\frac{k+1}{2}}y_{2,m-k}\\
&+ \sum_{k=0}^m  \floor{\frac{k+2}{2}}y_{0,m-k}+\left(2\floor{\frac{k+1}{2}} - \floor{\frac{k}{2}}\right)y_{1,m-k} + (k+1)y_{2,m-k} + \floor{\frac{k+1}{2}}y_{2,m-k}\\
&+\sum_{k=0}^m 2\left\lfloor\frac{k+2}{2}\right\rfloor y_{0,m-k} + 2\left\lfloor\frac{k+2}{2}\right\rfloor y_{1,m-k} + (2k+2) y_{2,m-k} + \left(2\left\lfloor\frac{k}{2}\right\rfloor + 1\right) y_{3,m-k}\\
&-\sum_{k=0}^m (k+1) y_{0,m-k} + (k+1) y_{1,m-k} + (2k+1) y_{2,m-k} + k y_{3,m-k}.
\end{align*}
Since the desired result is equivalent to (by evaluating the $k=0$ term and reindexing the sum)
\begin{align*}
a_{2,m+1} =&\ y_{0,m+1} + y_{1,m+1} + y_{2,m+1} \\ & + \sum_{k=0}^m (k+2) y_{0,m-k} + (k+2) y_{1,m-k} + (2k+3) y_{2,m-k} + (k+1) y_{3,m-k},
\end{align*}
the four identities needed to show that the result holds for $a_{2,m+1}$ are
\begin{subequations}
\begin{align}
k+2 &= 2\floor{\frac{k+1}{2}}-\floor{\frac{k}{2}} + \floor{\frac{k+2}{2}} + 2 \floor{\frac{k+2}{2}} - k - 1,\label{eq:pf43a3}\\
k+2 &= \floor{\frac{k+2}{2}} + 2\floor{\frac{k+2}{2}} - \floor{\frac{k}{2}} + 2\floor{\frac{k+2}{2}}-k-1,\label{eq:pf43b3}\\
2k+3 &= k+1+k+1+2k+2-2k-1,\label{eq:pf43c3}\\
k+1 &= \floor{\frac{k+1}{2}} + \floor{\frac{k+1}{2}} + 2\floor{\frac{k}{2}} + 1 -k.\label{eq:pf43d3}
\end{align}
\end{subequations}
The identity \eqref{eq:pf43c3} is trivial and \eqref{eq:pf43d3} is a direct application of Lemma~\ref{lemma:2}. To show \eqref{eq:pf43a3} and \eqref{eq:pf43b3} (which are both the same statement written different ways), note that
\begin{align*}
&2\floor{\frac{k+1}{2}}-\floor{\frac{k}{2}} + \floor{\frac{k+2}{2}} + 2 \floor{\frac{k+2}{2}} - k - 1\\
 &= 2\left(\floor{\frac{k+1}{2}} + \floor{\frac{k+2}{2}}\right) -\floor{\frac{k}{2}} + \floor{\frac{k+2}{2}} - k-1\\
&= 2k+2-\floor{\frac{k}{2}} + \floor{\frac{k+2}{2}} - k-1\\
 &= k+1-\floor{\frac{k}{2}} + \floor{\frac{k+2}{2}}.
\end{align*}
Thus, proving the statement reduces to showing that $-\floor{\frac{k}{2}} + \floor{\frac{k+2}{2}} = 1$. By Lemma~\ref{lemma:2}, $-\floor{\frac{k}{2}} = -k + \floor{\frac{k+1}{2}}$, so the statement further reduces to showing that $-k + \floor{\frac{k+1}{2}} + \floor{\frac{k+1}{2}} = 1$, which can be shown by another application of Lemma~\ref{lemma:2}. Therefore, each of the four desired identities holds, so the result holds for $a_{2,m+1}$.

Finally, to show the result for $a_{3,m+1}$, an identical process to the one used for $a_{2,m+1}$ can be used to see that the four desired identities are
\begin{subequations}
\begin{align}
2\floor{\frac{k+3}{2}} &= 2(k+2) - 2\floor{\frac{k+2}{2}},\label{eq:pf43a4}\\
2\floor{\frac{k+3}{2}} &= 2(k+2) - 2\floor{\frac{k+2}{2}},\label{eq:pf43b4}\\
2k+4 &= 2(2k+3) - 2k+2,\label{eq:pf43c4}\\
2\floor{\frac{k+1}{2}} +1 &= 2(k+1) - 2\floor{\frac{k}{2}}-1.\label{eq:pf43d4}
\end{align}
\end{subequations}
However, \eqref{eq:pf43c4} is trivial and each of the remaining statements are a straightforward applications of Lemma~\ref{lemma:2}, so each of these statements hold, and therefore so does the desired result for $a_{3,m+1}$.

Thus, by induction, the theorem holds.
\end{proof}

Note that Lemma~\ref{thm:5} shows that the $a_{i,k}$ are uniquely determined.  This gives the following theorem.

\begin{thm}\label{prop:2}
Let $M = \prod_{i\in I} \prod_{k \geq 0} Y_{i,k}^{y_{i,k}} \one\in \MM(\infty)$. Then 
\[
D(M) = 
\sum_{m \geq 0}\left( \sum_{k=0}^m \left(2\left\lfloor\frac{k}{2}\right\rfloor - \left\lfloor\frac{k-1}{2}\right\rfloor\right) y_{0,m-k} + \left\lfloor\frac{k+1}{2}\right\rfloor y_{1,m-k} + k y_{2,m-k} + \left\lfloor\frac{k}{2}\right\rfloor y_{3,m-k} \right).
\]
\end{thm}

\begin{ex}
Note that $\f_0 \f_1 \f_2 \f_3\mathbf 1 = Y_{0,3}^{-1} Y_{1,3}^{-1} Y_{2,2} Y_{3,0}^{-1} Y_{3,1}\one$ and
\[
\wt(\f_0 \f_1 \f_2 \f_3\one) = -\Lambda_0 - \Lambda_1 + \Lambda_2 - \delta.
\]
Indeed, consider the values of $a_{0,m}$ for each $m$. For $m \geq 3$, $a_{0,m} = 0$ since the largest nonzero $y_{i,k}$ is $y_{1,3}$. Now, applying Lemma~\ref{thm:5} (and ignoring the values for which $y_{i,k}=0$) shows that
\begin{align*}
a_{0,0} &= \floor{\frac{0}{2}}y_{3,0} = 0,\\
a_{1,0} &= \floor{\frac{0}{2}}y_{3,1} + \floor{\frac{1}{2}}y_{3,0} = 0,\\
a_{2,0} &= 0y_{2,2} + \floor{\frac{1}{2}}y_{3,1} + \floor{\frac{2}{2}}y_{3,0} = -1.
\end{align*}
Thus, by adding these up, Theorem~\ref{prop:2} implies $D(\f_0 \f_1 \f_2 \f_3\one) = -1$, as expected.
\end{ex}

\subsection{Comments on type $B_4^{(1)}$}

In an attempt to find a similar result for the $U_q(B_4^{(1)}$)-crystal $\MM(\infty)$, note that the analogous defining identities to Lemma \ref{lemma:1} are
\begin{align*}
a_{0,m} &= y_{0,m} + a_{2,m-1} - a_{0, m-1},\\
a_{1,m} &= y_{1,m} + a_{2,m-1} - a_{1, m-1},\\
a_{2,m} &= y_{2,m} + a_{0,m} + a_{1,m} + a_{3,m-1} - a_{2,m-1},\\
a_{3,m} &= y_{3,m} + a_{2,m} + a_{4,m-1} - a_{3,m-1},\\
a_{4,m} &= y_{4,m} + 2a_{3,m} - a_{4,m-1}.
\end{align*}
Note that, if
\[
a_{0,m} = \sum_{k=0}^m a_k y_{0,m-k} +  b_k y_{1,m-k} + c_k y_{2,m-k} +  d_k y_{3,m-k} + e_k y_{4,m-k}
\]
(and given similar recurrence identities to each $a_{i,m}$), the first $k$ terms of each sequence can be manually computed. This can be done by first noting that 
\begin{align*}
a_{0,0} &= y_{0,0}, &a_{3,0} &= y_{0,0} + y_{1, 0} + y_{2,0} + y_{3,0}, \\ 
a_{1,0} &= y_{1,0},& a_{4, 0} &= 2y_{0,0} + 2y_{1,0} + 2y_{2,0} + 2y_{3,0} + y_{4,0},\\
a_{2,0} &= y_{0,0} + y_{1,0} + y_{2,0}.
\end{align*}
Then, the known first terms of each sequence can be plugged into the analogous Lemma~\ref{lemma:1} identities to generate each coefficient. The following code can be used in \sage\ \cite{Sage} to compute the first 21 values of the sequences $(a_k)_{k=1}^\infty$ and $(b_k)_{k=1}^\infty$: 
\begin{lstlisting}
sage: def coefficientsB4(n):
....:     a = vector([1,0,0,0,0])
....:     b = vector([0,1,0,0,0])
....:     c = vector([1,1,1,0,0])
....:     d = vector([1,1,1,1,0])
....:     e = vector([2,2,2,2,1])
....:     print['k=',0, 'a_k=', a[0], 'b_k=',a[1]]
....:     for i in range(n):
....:         a = c-a
....:         b = c-b
....:         c = a+b-c+d
....:         d = c+e-d
....:         e = 2*d-e
....:         print ['k=',i+1, 'a_k=', a[0], 'b_k=',a[1]]
....:      
sage: coefficientsB4(20)
\end{lstlisting} 
\begin{align*}
(a_k)_{k=0}^{20} &= (1,0,1,1,2,1,3,2,3,3,4,3,5,4,5,5,6,5,7,6,7), \\
(b_k)_{k=0}^{20} &= (0,1,0,2,1,2,2,3,2,4,3,4,4,5,4,6,5,6,6,7,6).
\end{align*}
Note that each of these sequences is not as simple as the sequences needed for $B_3^{(1)}$. Therefore, while the same method of finding sequences that generate the coefficients may work here, it is not immediately apparent how they would do so. In particular, the Online Encyclopedia of Integer Sequences \cite{OEIS} notes that these first terms of $a_k$ are consistent with the power series expansion of $\frac{1+x^4}{(1-x^2)(1-x^3)}$. However, the sequence $(b_k)_{k=0}^{20}$ was not recognized by the Online Encyclopedia of Integer Sequences. 

\section{Irreducible highest weight crystals}\label{sec:hw}

Define the $\MM$ to be the set of all monomials of the form
\[
M = \prod_{i \in I} \prod_{k \geq 0} Y_{i,k}^{y_{i,k}},
\]
where $y_{i,k} \in \ZZ$ and $y_{i,k} = 0$ for all but finitely many $k$.  The differences between $\MM$ and $\hM$ as sets is the inclusion of the variable $\one$ in $\hM$.  the definition of $\varepsilon_i$, and the definition of $\f_i$.  A crystal structure may be defined on $\MM$ using the same operations from Equation \eqref{eq:mon_cry_ops}, except replacing $\varepsilon_i(M)$ in \eqref{eq:oldep} with 
\[
\varepsilon_i(M) = \max\left\{ - \sum_{j>k} y_{i,j} : k \in \ZZ \right\}
\]
and replacing the definition of $\f_i$ in \eqref{eq:infop} by
\[
\f_i M = \begin{cases} 0 & \text{ if } \varphi_i(M) = 0, \\
A_{i,k_f}^{-1}M & \text{ if } \varphi_i(M) >0, \end{cases}
\]
Kashiwara~\cite{Kash:03} proved that if $M \in \MM$ is a monomial of weight $\lambda$ such that $\e_iM = 0$ for all $i\in I$, then the connected component of $\MM$ containing $M$ is isomorphic to the irreducible highest weight crystal $B(\lambda)$.  However, just as in the case of $B(\infty)$ above, if $\g$ is of affine type, then the two crystals are isomorphic as $U_q'(\g)$-crystals rather than $U_q(\g)$-crystals.

For consistency, if $\lambda = \sum_{i\in I} p_i \Lambda_i$ is a dominant integral weight, define $H_{\lambda} = \prod_{i \in I} Y_{i,0}^{p_i}$.  Direct calculations show that $\e_iH_\lambda =0 $ for all $i\in I$ and that $\wt(H_\lambda) = \lambda$.  Henceforth, denote the connected component of $\MM$ containing $H_\lambda$ by $\MM(\lambda)$.  Moreover, the morphism $\MM(\lambda) \lhook\joinrel\longrightarrow \MM(\infty)\otimes T_\lambda$ defined by $M \mapsto H_\lambda^{-1}M\one \otimes t_\lambda$ is a full crystal embedding.

\begin{ex}
Consider the realization $\MM(2\Lambda_1)$ of the irreducible highest weight crystal $B(2\Lambda_1)$ in type $A_2^{(1)}$.  Choose $Y_{1,0}^2$ to be the monomial of weight $2\Lambda_1$ to generate this crystal.  Then 
\[
M = \f_2\f_0\f_1 Y_{1,0}^2 = Y_{1,0}Y_{1,3}Y_{2,0}Y_{2,3}^{-1}.
\]
Using the crystal axioms, we have $\wt(M) = 2\Lambda_1-\delta$, but using the definition of the weight function for Nakajima monomials we get $\wt(M) = 2\Lambda_1$.  In this example, there are no variables of the form $Y_{0,k}$ in the expression for $M$.  However, 
\[
Y_{1,0}^{-2}M\one = Y_{1,0}^{-1}Y_{1,3}Y_{2,0}Y_{2,3}^{-1}\one = A_{0,1}^{-1} A_{1,0}^{-1} A_{2,2}^{-1}\one,
\]
which has weight $-\delta$ in $\MM(\infty)$; that is, $\wt(Y_{1,0}^{-2}M\one \otimes t_{2\Lambda_1}) = 2\Lambda_1 - \delta$.
\end{ex}

\begin{prop}
Let $\lambda$ be a dominant integral weight for $\g$ of affine type.  If $D\colon \MM(\infty) \longrightarrow \ZZ$ is the map from Problem \ref{prob:main}, then $\wt\colon \MM(\lambda) \longrightarrow P$ is defined by
\[
\wt(M) = \sum_{i\in I} \Bigl( \sum_{k\ge 0} y_{i,k} \Bigr) \Lambda_i + D(H_\lambda^{-1}M\one)\delta,
\]
where $M = \prod_{i \in I} \prod_{k \geq 0} Y_{i,k}^{y_{i,k}}$.
\end{prop}

\begin{proof}
Embed $\MM(\lambda) \lhook\joinrel\longrightarrow \MM(\infty) \otimes T_{\lambda}$ using the map $M \mapsto H_\lambda^{-1}M\one \otimes t_\lambda$.
\end{proof}

\begin{acknowledgements}
The authors would like to thank Sidney Graham and Meera Mainkar for their comments on an earlier version of this manuscript which was the first author's Master's thesis.  The authors would also like to thank Jeong-Ah Kim, Kyu-Hwan Lee, Travis Scrimshaw, and Dong-Uy Shin for valuable conversations, detailed comments, and encouragement.  
\end{acknowledgements}

\bibliography{monomial_affine_weights}{}

\providecommand{\bysame}{\leavevmode\hbox to3em{\hrulefill}\thinspace}
\begin{thebibliography}{10}

\bibitem{Sage}
The~Sage Developers, \emph{{S}age {M}athematics {S}oftware ({V}ersion 7.6)},
  The Sage Development Team, 2017, \url{http://www.sagemath.org}.

\bibitem{HK:02}
Jin Hong and Seok-Jin Kang, \emph{Introduction to quantum groups and crystal
  bases}, Graduate Studies in Mathematics, vol.~42, American Mathematical
  Society, Providence, RI, 2002.

\bibitem{Kac:90}
Victor~G. Kac, \emph{Infinite-dimensional {L}ie algebras}, third ed., Cambridge
  University Press, Cambridge, 1990.

\bibitem{KKS:07}
Seok-Jin Kang, Jeong-Ah Kim, and Dong-Uy Shin, \emph{Modified {N}akajima
  monomials and the crystal {$B(\infty)$}}, J. Algebra \textbf{308} (2007),
  no.~2, 524--535.

\bibitem{Kash:91}
Masaki Kashiwara, \emph{On crystal bases of the {$q$}-analogue of universal
  enveloping algebras}, Duke Math. J. \textbf{63} (1991), no.~2, 465--516.

\bibitem{Kash:93}
\bysame, \emph{The crystal base and {L}ittelmann's refined {D}emazure character
  formula}, Duke Math. J. \textbf{71} (1993), no.~3, 839--858.

\bibitem{Kash:03}
\bysame, \emph{Realizations of crystals}, Combinatorial and geometric
  representation theory ({S}eoul, 2001), Contemp. Math., vol. 325, Amer. Math.
  Soc., Providence, RI, 2003, pp.~133--139.

\bibitem{KS:10}
Jeong-Ah Kim and Dong-Uy Shin, \emph{Generalized {Y}oung walls and crystal
  bases for quantum affine algebra of type {$A$}}, Proc. Amer. Math. Soc.
  \textbf{138} (2010), no.~11, 3877--3889.

\bibitem{Nak:03}
Hiraku Nakajima, \emph{{$t$}-analogs of {$q$}-characters of quantum affine
  algebras of type {$A_n,D_n$}}, Combinatorial and geometric representation
  theory ({S}eoul, 2001), Contemp. Math., vol. 325, Amer. Math. Soc.,
  Providence, RI, 2003, pp.~141--160.

\bibitem{OEIS}
The On-Line~Encyclopedia of~Integer~Sequences, 2017, Published electronically
  at \url{http://oeis.org/A051274}.

\end{thebibliography}
\bibliographystyle{amsplain}

\end{document}